\documentclass[12pt,twoside]{amsart}
\usepackage{amsmath}
\usepackage{amssymb}
\usepackage{amscd}
\usepackage{xypic}
\xyoption{all}
\title[Positivity of Hochster Theta over $\mathbb{C}$]{Positivity of Hochster Theta over $\mathbb{C}$}

\author{Mohammad Reza Rahmati}
\thanks{}
\address{Centro de Investigacion en Matematicas, A.C.
\hfill\break 
\hfill\break \\
\hfill\break }
\email{mrahmati@cimat.mx}

\newcommand{\comments}[1]{}

\newtheorem{theorem}{Theorem}[section]

\newtheorem{corollary}[theorem]{Corollary}

\newtheorem{definition}[theorem]{Definition}

\newtheorem{example}[theorem]{Example}



\begin{document}

\begin{abstract}
M. Hochster defines an invariant namely $\Theta(M,N)$ associated to two finitely generated module over a hyper-surface ring $R=P/f$, where $P=k\{x_0,...,x_n\}$ or $k[X_0,...,x_n]$, for $k$ a field and $f$ is a germ of holomorphic function or a polynomial, having isolated singularity at $0$. This invariant can be lifted to 
the Grothendieck group $G_0(R)_{\mathbb{Q}}$ and is compatible with the chern character and cycle class map, according to the works of W. Moore, G. Piepmeyer, S. Spiroff, M. Walker. They prove that it is semi-definite when $f$ is a homogeneous polynomial, using Hodge theory on Projective varieties. It is a conjecture that the same holds for general isolated singularity $f$. We give a proof of this conjecture using Hodge theory of isolated hyper-surface singularities when $k=\mathbb{C}$. We apply this result to give a positivity criteria for intersection multiplicty of proper intersections in the variety of $f$. 
\end{abstract}

\maketitle


\section*{Introduction}

\vspace{0.5cm}

One of the important ways to generalize known facts about smooth algebraic varieties in algebraic geometry, is to try to extend them over singular varieties. Most of the valid methods in the smooth category are hard to be worked out in the singular category. An example of this is the definition of intersection multiplicity as the $Tor$-formula in algebraic geometry as an Euler characteristic. The Riemann-Roch formula provides a definition of this characteristic class by the chern character map on the $K$-group. The useful observation here is the intersection theory is in fact a theory on $K$-groups. In the singular set up both of the definitions of $K$-groups and the chern character become very complicated in the first glance. although trying to generalize the homology theory over singular varieties provides some type of solution for this case, trying to extend the natural concepts from the smooth category to the singular one remains still complicated and difficult. 

\vspace{0.5cm}

In this short note we provide a very basic example of this situation and try to approach to the solution using asymptotic Hodge theory. In this way we employ some standard methods in the extensions of polarized mixed Hodge structure to answer questions about singularities in algebraic geometry. Thus the reader may divide the text into three part, the first is a singular set up in algebraic geometry, the second is on variation of mixed Hodge structures, and the third is an application of the results mentioned in part 2.     

\vspace{0.5cm}

\section{Modules over Hypersurface rings}

\vspace{0.5cm}

A hyper-surface ring is a ring of the form $R:=P/(f)$, where $P$ is an arbitrary ring and $f$ a non-zero divisor. Localizing we may assume $P$ is a local ring of dimension $n+1$. As according to the title we assume $P=\mathbb{C}\{x_0,...,x_n\}$ and $f$ a holomorphic germ, or $P=\mathbb{C}[x_0,...,x_n]$ and then $f$ would be a polynomial. Then we are mainly interested to study finitely generated modules over these rings. Consider $f:\mathbb{C}^{n+1} \to \mathbb{C}$ in this form, and choose a representative for the Milnor fibration as $f:X \to T$, where $T$ is the disc around $0$. 

\vspace{0.5cm}

\noindent
\textit{Then, through all the rest of this text we assume $0 \in \mathbb{C}^{n+1}$ is the only singularity of $f$}. 

\vspace{0.5cm}

A matrix factorization of $f$ in $P$ is a pair of matrices $A$ and $B$ such that $AB=BA=f.\ id$. It is equivalent to the data of a pair of finitely generated free $P$-modules 

\[ d_0:X^0 \leftrightarrows X^1:d_1 , \qquad d_0d_1=d_1d_0=f. \ id \] 

\vspace{0.5cm}

\noindent
It is a basic fact, discovered by D. Eisenbud, that the $R$-modules have a minimal resolution that is eventually 2-periodic. Specifically, in a free resolution of such a module $M$, we see that after n-steps we have an exact sequence of the following form. 

\begin{equation}
0 \to M' \to F_{n-1} \to F_{n-2} \to ... \to F_0 \to M \to 0
\end{equation}

\vspace{0.5cm}

\noindent
where the $F_i$ are free $R$-modules of finite rank and $depth_R(M')=n$. If $M'=0$ then $M$ has a free resolution of finite length., If $M' \ne 0$, then $M'$ is a maximal Cohen-Macaulay module, that is $depth_R(M')=n$. So ''up to free modules'' any $R$-module can be replaced by 
a maximal Cohen-Macaulay module. If $M$ is a maximal Cohen-Macaulay $R$-module that is minimally generated by $p$ elements, its resolution as $P$-module has the form

\vspace{0.5cm}

\begin{center}

$\begin{array}[c]{ccccccccc}
0 &\rightarrow &P^p &\stackrel{A}{\rightarrow} &P^p &\rightarrow &M &\rightarrow& 0\\
&&\downarrow&\stackrel{B}{\swarrow}&\downarrow&&\downarrow{0}&&\\
0 &\rightarrow &P^p &\stackrel{A}{\rightarrow} &P^p &\rightarrow &M &\rightarrow& 0
\end{array}$

\end{center}

\vspace{0.5cm}

\noindent
where $A$ is some $p \times p$ matrix with $det(A)=f^q$. The fact that multiplication by $f$ acts as $0$ on $M$ produces a matrix $B$ such that $A.B=B.A=f.I$, where $I$ is the identity matrix. In other words, we find a matrix factorization $(A,B)$ of $f$ determined uniquely up to base change in the free module $P^p$, by $M$. The matrix factorization not only determines $M$ but also a resolution of $M$ as $R$-module. 

\vspace{0.5cm}

\begin{center}
$....\to R^p \to R^p \to R^p \to M \to 0$.
\end{center}

\vspace{0.5cm}

\noindent
So a minimal resolution of $M$ looks in general as follows

\vspace{0.5cm}

\begin{center}
$... \to G \to F \to G \to F_{n-1} \to ... \to F_0 \to M \to 0$.
\end{center}

\vspace{0.5cm}

\noindent
As a consequence all the homological invariants like $Tor_k^R(M,N), Ext_R^k(M,n)$ are 2-priodic, \cite{BVS}, \cite{EP}. 

\vspace{0.5cm}

The category of matrix factorizations of $f$ over $R$, namely $MF(R,f)$; is defined to be the differential $\mathbb{Z}/2$-graded category, whose objects are pairs $(X,d)$, where $X=X^0 \oplus X^1$ is a free $\mathbb{Z}/2$-graded $R$-module of finite rank equipped with an $R$-linear map $d$ of odd degree satisfying $d^2=f.\ id_X$. Here the degree is calculated in $\mathbb{Z}/2$. Regarding to the first definition 

\[ d= \left( 
\begin{array}{cc}
0  &  d_0\\
d_1 &   0
\end{array} \right), \qquad d^2=f.\ id. \]

\vspace{0.5cm}

\noindent
The morphisms $MF(X,X')$ are given by $\mathbb{Z}/2$-graded $R$-module maps from $X$ to $X'$ (or equivalent between the components $X^0$ and $X^1$) provided that the differential is given by
 
\begin{equation}
d(f)=d_{X'} \circ f - (-1)^{\mid f \mid} f \circ d_X.
\end{equation}

\vspace{0.5cm}

\noindent
Here $d_{X}$ or $d_X^{\prime}$ may be considered as the matrix given above or to be separately $d_0$ and $d_1$, and also it is evident that $d(f)^2=0$. By choosing bases for $X^0$ and $X^1$ we reach to the former definition, \cite{EP}. 

\vspace{0.5cm}

\section{Hochster Theta Function}

\vspace{0.5cm}

M. Hochster in his study of direct summand conjecture defined the following invariant for the hypersurface ring $R=P/(f)$, namely $\Theta$-invariant.

\vspace{0.5cm}
 
\begin{definition}(Hochster Theta pairing)
The theta pairing of two $R$-modules $M$ and $N$ over a hyper-surface ring
$R/(f)$ is 

\vspace{0.5cm}

\begin{center}
$\Theta(M,N):= l(Tor_{2k}^R(M,N))-l(Tor_{2k+1}^R(M,N)), \qquad k>>0$
\end{center}

\vspace{0.5cm}

\noindent
The definition makes sense as soon as the lengths appearing are finite. This certainly happens if $R$ has an isolated singular point. 
\end{definition}

\vspace{0.5cm}

\begin{example} \cite{BVS} Take $f=xy-z^2, M=\mathbb{C}[[xyz]]/(x,y)$. A matrix factorization $(A,B)$ associated to $M$ is given by \\
\[ A=\left( \begin{array}{cc}
y  & -z \\
-z &  x 
\end{array} \right), \qquad
 B=\left( \begin{array}{cc}
x &  z \\
z &  y 
\end{array} \right).\]

\vspace{0.5cm}

\noindent
The $Tor_k^R(M,M)$ is the homology of the complex

\vspace{0.5cm}

\begin{center}
$.....\to \mathbb{C}[[y]]^2 \to \mathbb{C}[[y]]^2 \to \mathbb{C}[[y]] \to 0$
\end{center}

\vspace{0.5cm}

\noindent
where\\
\[ \alpha=\left( \begin{array}{cc}
y & 0 \\
0 & 0
\end{array} \right), \qquad
 \beta=\left( \begin{array}{cc}
0 & 0 \\
0 & y 
\end{array} \right).\]\\ 

\vspace{0.3cm}

So we find that $\Theta(M,M)=0$.
\end{example}

\vspace{0.5cm}

Hochster theta pairing is additive on short exact sequences in each argument, and thus determines a $\mathbb{Z}$-valued pairing on $G(R)$, the Grothendieck group of finitely generated $R$-modules. One loses no information by tensoring with $\mathbb{Q}$ and so often theta is interpreted as a symmetric bilinear form on the rational vector space $G(R)_{\mathbb{Q}}$. It is basic that Theta would vanish if either $M$ or $N$ be Artinian or have finite projective dimension \cite{MPSW}, \cite{BVS}. The $\Theta$-invariant has different interpretations as intersection multiplicity in the singular category. 

\vspace{0.5cm}
 
\begin{theorem} \cite{BVS} When $M=\mathcal{O}_Y=R/I,N=\mathcal{O}_Z=R/J$, where $Y,Z \subseteq X_0$ are the sub-varieties defined by the ideals $I,J$ respectively, then

\vspace{0.5cm}
 
\begin{center}
$\Theta(\mathcal{O}_Y,\mathcal{O}_Z)=i(0;Y,Z)$ 
\end{center}

\vspace{0.5cm}

\noindent
in case that $Y \cap Z={0}$. Here $i(0; , )$ is the ordinary intersection multiplicity in $\mathbb{C}^{n+1}$.
\end{theorem}

\vspace{0.5cm}

\noindent
By additivity over short exact sequences and the fact that any module admits a finite filtration with sub-quotients of the form $R/I$, knowing $\Theta(\mathcal{O}_Y,\mathcal{O}_Z)$ determines
$\Theta(M,N)$ for all modules $M,N$.

\vspace{0.5cm}

\begin{theorem} \cite{BVS} Assume $f \in \mathbb{C}[[x_1,...,x_{2m+2}]$ is a homogeneous polynomial of degree $d$, and $X_0=f^{-1}(0) \in \mathbb{C}^{2m+2}$ and $T=V(f) \in \mathbb{P}^{2m+1}$ the associated projective cone of degree $d$. Let $Y$ and $Z$ be also co-dimension $m$ cycles in $T$. If $Y,Z$ intersect transversely, then

\vspace{0.5cm}
 
\begin{center}
$\Theta(\mathcal{O}_Y,\mathcal{O}_Z)=-\frac{1}{d}[[Y]].[[Z]]$
\end{center}

\vspace{0.5cm}

\noindent
where $[[Y]]:=d[Y]-deg(Y).h^m$ is the primitive class of $[Y]$, with $h \in H^1(T)$ the hyperplane class.
\end{theorem}

\vspace{0.5cm}

\noindent
The primitive class of a cycle $Y$ is the projection of its fundamental class $[Y] \in H^m(T)$ into the orthogonal complement to $h^m$ with respect to the intersection pairing into $H^{2m}(T)=\mathbb{C}$. As $h^m. h^m=d=deg(T)$ and $[Y].h^m=deg(Y)$ the description of the primitive class follows. Substituting the claim can be reformulated

\vspace{0.5cm}

\begin{center}
$\Theta(\mathcal{O}_Y,\mathcal{O}_Z)=-\frac{1}{d}[[Y]].[[Z]]=-d[Y].[Z]+deg(Y)deg(Z)$
\end{center}

\vspace{0.5cm}

\noindent
Where $[Y].[Z]$ denotes the intersection form on the cohomology of the projective space, \cite{BVS}.

\vspace{0.5cm} 

When $f$ in consideration is a homogeneous polynomial of degree $d$, such that $X:=Proj(R)$ is a smooth $k$-variety, the Theta pairing is induced, via chern character map, from the pairing on the primitive part of de Rham cohomology

\vspace{0.5cm}

\begin{center}
$\displaystyle{\frac{H^{(n-1)/2}(X,\mathbb{C})}{\mathbb{C}. \gamma^{(n-1)/2}} \times \frac{H^{(n-1)/2}(X,\mathbb{C})}{\mathbb{C}. \gamma^{(n-1)/2}} \to \mathbb{C}}$
\end{center}

\vspace{0.5cm}

\noindent
given by 

\vspace{0.5cm}

\begin{center}
$(a,b) \to (\int_{X} a \cup \gamma^{(n-1)/2})(\int_{X} a \cup \gamma^{(n-1)/2})-d(\int_{X} a \cup b)$
\end{center}

\vspace{0.5cm}

\noindent
where $\gamma$ is the class of a hyperplane section and Theta would vanish for rings of this type having even dimensions. When $n=1$ by $\gamma^0$ we mean $1 \in H^0(X,\mathbb{C})$, \cite{MPSW}.

\vspace{0.5cm}

\begin{theorem} \cite{MPSW} 
For $R$ and $X$ as above, if $n$ is odd there is a commutative diagram

\begin{equation}
\begin{CD}
G(R)_{\mathbb{Q}}^{\otimes 2}  @<<<  \displaystyle{(\frac{K(X)_{\mathbb{Q}}}{\alpha})^{\otimes 2}}\\
@V{\Theta}VV        @VV{(ch^{n-1/2})^{\otimes 2}}V\\
\mathbb{C}  @<<{\theta}<  ( \displaystyle{\frac{H^{(n-1)/2}(X,\mathbb{C})}{\mathbb{C}. \gamma^{(n-1)/2}}} )^{\otimes 2}
\end{CD} 
\end{equation}

\vspace{0.5cm}

\end{theorem}

\vspace{0.5cm}

\begin{theorem}  \cite{MPSW} For $R$ and $X$ as above and $n$ odd the restriction of the pairing $(-1)^{(n+1)/2} \Theta $ to

\vspace{0.5cm}
 
\begin{center}
$im(ch^{\frac{n-1}{2}}): K(X)_{\mathbb{Q}}/{\alpha} \to \displaystyle{\frac{H^{(n-1)/2}(X,\mathbb{C})}{\mathbb{C}. \gamma^{\frac{n-1}{2}}}}$ 
\end{center}

\vspace{0.5cm}

\noindent
is positive definite. i.e. $(-1)^{(n+1)/2}\Theta(v,v) \geq 0$
with equality holding if and only if $v=0$. In this way $\theta$ is semi-definite on $G(R)$.
\end{theorem}

\begin{proof} \cite{MPSW} Define 

\vspace{0.5cm}

\begin{center}
$W=H^{n-1}(X(\mathbb{C}),\mathbb{Q}) \cap H^{\frac{n-1}{2},\frac{n-1}{2}}(X(\mathbb{C}))$.
\end{center}

\vspace{0.4cm}

\noindent
It is classical that the image of $ch^{(n-1)/2}$ is contained in $W$. Define $e:W/\mathbb{Q}. \gamma^{(n-1)/2} \hookrightarrow H^{n-1}(X,\mathbb{Q})$ by

\vspace{0.3cm}

\begin{center}
$e(a)=a-\displaystyle{\frac{\int_X a \cup \gamma^{(n-1)/2}}{d}. \gamma^{(n-1)/2}} \in W$.
\end{center}

\vspace{0.3cm}

\noindent
We know that $\theta(a,b)=-d.I^{coh}(e(a),e(b))$ Now the theorem follows from the polarization properties of cup product on cohomology of projective varieties.
\end{proof}

\vspace{0.5cm}

\section{Hodge theory and Residue form}

\vspace{0.5cm}

Asuume $f:\mathbb{C}^{n+1} \to \mathbb{C}$ is a germ of isolated singularity. We choose a representative $f:X \to T$ over a small disc $T$ according to the milnor fibration theorem. It is possible to embed the Milnor fibration $f:X \to T$ into a compactified (projective) fibration $f_Y : Y \to T$ such that the fiber $Y_t$ sits in $\mathbb{P}^{n+1}$ for $t \neq 0$. The projective fibration $f_Y$ has a  unique singularity at $0 \in Y_0$ over $t = 0$. Then, there exists a short exact sequence
 
\begin{equation}
0 \to H^n(Y_0,\mathbb{Q}) \to H^n(Y_t,\mathbb{Q}) \stackrel{i^*}{\rightarrow} H^n(X_t,\mathbb{Q}) \to 0 , \qquad t \neq 0.
\end{equation}

\vspace{0.5cm}

We have $H^n(Y_0,\mathbb{Q})=\ker(M_Y -id)$, by the invariant cycle theorem, where $M_Y$ is the monodromy of $f_Y$. The form $S_Y :=(-1)^{n(n-1)/2} I_Y^{coh} : H^n(Y_t,\mathbb{Q}) \times H^n(Y_t,\mathbb{Q}) \to \mathbb{Q}$ is the polarization form of pure Hodge structure on $H^n(Y_t,\mathbb{C})$, $t \in T^{\prime}$ (the punctured disc). W. Schmid (resp. J. Steenbrink) has defined a canonical MHS on $H^n(Y_t,\mathbb{Q})$ (resp. on $H^n(Y_t,\mathbb{Q})$) namely limit MHS, which make the above sequence an exact sequence of MHS's.

\vspace{0.5cm}

\noindent
In the short exact sequence, the map $i^*$ is an isomorphism on $H^n(Y_t,\mathbb{Q})_{\neq 1} \to H^n(X_t,\mathbb{Q})_{\neq 1}$ giving $S=(-1)^{n(n-1)/2} I^{coh} =(-1)^{n(n-1)/2} I_Y^{coh}= S_Y$ on  $H^n(X_t,\mathbb{Q})_{\neq 1}$.\\[0.3cm]
The above short exact sequence restricts to the following one, 

\begin{equation}
0 \to \ker \{N_Y:H^n(Y_t,\mathbb{Q})_{1} \to H^n(Y_t,\mathbb{Q})_{1}\} \to H^n(Y_t,\mathbb{Q})_{1} \to H^n(X_t,\mathbb{Q})_{1} \to 0
\end{equation}

\vspace{0.5cm}

\noindent
So $a,b \in H^n(X_t,\mathbb{Q})_{1}$ have pre-images $a_Y , b_Y \in H^n(Y_t,\mathbb{Q})_{1}$
and 
\begin{equation}
S(a,b)=S_Y(a_Y,(-N_Y)b_Y) 
\end{equation}

\vspace{0.2cm}

\noindent
is independent of the lifts of $a_Y , b_Y$, by the fact that $N_Y$ is an infinitesimal isometry for $S_Y$. The equation (10) defines the desired polarization on $H^n(X_t,\mathbb{Q})_1$. The polarization form $S$ is $M$-invariant, non-degenerate, $(-1)^n$-symmetric on $H^n(X_t,\mathbb{Q})_{\neq 1}$ and $(-1)^{n+1}$-symmetric on $H^n(X_t,\mathbb{Q})_{1}$, \cite{H1}.

\vspace{0.5cm}

\noindent
Suppose,
\begin{center}
$H^n(X_{\infty}, \mathbb{C})= \displaystyle{\bigoplus_{p,q,\lambda}} I^{p,q}_{\lambda}$ 
\end{center}

\vspace{0.5cm}

\noindent
is the Deligne-Hodge bigrading, and generalized eigen-spaces of vanishing cohomology, and also $\lambda=\exp(-2\pi i \alpha)$ with $\alpha \in (-1,0]$. Consider the isomorphism obtained by composing the three maps,

\begin{equation}
\Phi_{\lambda}^{p,q}: I^{p,q}_{\lambda} \stackrel{\hat{\Phi}_{\lambda}}{\longrightarrow}Gr_V^{\alpha+n-p}H'' \stackrel{pr}{\longrightarrow} Gr_V^{\bullet} H^{\prime \prime}/\partial_t^{-1}H^{\prime \prime} \stackrel{\cong}{\longrightarrow} \Omega_f
\end{equation}

\vspace{0.5cm}

\noindent
where 

\vspace{0.5cm}

\begin{center}
$\hat{\Phi}_{\lambda}^{p,q}:= \partial_t^{p-n} \circ \psi_{\alpha}|  I^{p,q}_{\lambda}$ \\[0.2cm] 
$\Phi=\bigoplus_{p,q,\lambda} \Phi_{\lambda}^{p,q}, \qquad \Phi_{\lambda}^{p,q}=pr \circ \hat{\Phi}_{\lambda}^{p,q}$

\end{center}

\vspace{0.5cm}

\noindent
where $I^{p,q}$ stands for the bigrading, $\partial_t$ is the Gauss-Manin connection and $\psi_{\alpha}$ is the nearby map cf. \cite{R}. The map $\Phi$ is obviously an isomorphism because both of the $\psi_{\alpha}$ and $\partial_t^{-1}$ are isomorphisms, cf. \cite{H1}, \cite{V}.

\vspace{0.5cm}

\begin{definition}(MHS on $\Omega_f$)
We define a mixed Hodge structure on $\Omega_f$ using the isomorphism $\Phi$. This means that all the data of the Steenbrink MHS on $H^n(X_{\infty},\mathbb{C})$ such as the $\mathbb{Q}$ or $\mathbb{R}$-structure, the weight filtration $W_{\bullet}\Omega_{f,\mathbb{Q}}$ and the Hodge filtration $F^{\bullet}\Omega_{f,\mathbb{C}}$ is defined via the isomorphism $\Phi$. Specifically; in this we obtain a conjugation map

\begin{equation}
\bar{.}:\Omega_{f,\mathbb{Q}} \otimes \mathbb{C} \to \Omega_{f,\mathbb{Q}} \otimes \mathbb{C}, \qquad \Omega_{f,\mathbb{Q}}:=\Phi^{-1}H^n(X_{\infty},\mathbb{Q})
\end{equation}

\vspace{0.5cm}

\noindent
defined from the conjugation on $H^n(X_{\infty},\mathbb{C})$ via this isomorphism. 

\end{definition}

\vspace{0.5cm}

\noindent
Recall that the limit (Steenbrink) mixed Hodge structure, 
is defined by

\[ F^p H^n(X_{\infty},\mathbb{C})_{\lambda}=\psi_{\alpha}^{-1}(Gr_V^{\alpha} \partial_t^{n-p} H^{\prime \prime}) \] 

\vspace{0.5cm}

\noindent
This justifies the power of $\partial_t^{-1}$ applied in the definition of $\Phi$. 

\vspace{0.5cm}

\begin{theorem}  \cite{R} 
Assume $f:(\mathbb{C}^{n+1},0) \to (\mathbb{C},0)$, is a holomorphic germ with isolated singularity at $0$. Then, the isomorphism $\Phi$ makes the following diagram commutative up to a complex constant;

\begin{equation}
\begin{CD}
\widehat{Res}_{f,0}:\Omega_f \times \Omega_f @>>> \mathbb{C}\\
@VV(\Phi^{-1},\Phi^{-1})V                   @VV \times *V \\
S:H^n(X_{\infty}) \times H^n(X_{\infty}) @>>> \mathbb{C}
\end{CD} \qquad \qquad  * \ne 0
\end{equation}

\vspace{0.5cm}

\noindent
where, 

\[ \widehat{Res}_{f,0}=\text{res}_{f,0}\ (\bullet,\tilde{C}\ \bullet) \]

\vspace{0.5cm}

\noindent
and $\tilde{C}$ is defined relative to the Deligne decomposition of $\Omega_f$, via the isomorphism $\Phi$. If $J^{p,q}=\Phi^{-1} I^{p,q}$ is the corresponding subspace of $\Omega_f$, then

\begin{equation}
\Omega_f=\displaystyle{\bigoplus_{p,q}}J^{p,q} \qquad \tilde{C}|_{J^{p,q}}=(-1)^{p} 
\end{equation}

\vspace{0.5cm}

\noindent
In other words;

\vspace{0.1cm}

\begin{equation}
S(\Phi^{-1}(\omega),\Phi^{-1}(\eta))= * \times \ \text{res}_{f,0}(\omega,\tilde{C}.\eta), \qquad 0 \ne * \in \mathbb{C}
\end{equation}

\end{theorem} 

\vspace{0.5cm}

\noindent
The proof of the Theorem 3.2 is a generalization of an argument in \cite{CIR}, in the quasi-homogeneous case. 

\vspace{0.5cm}

\begin{corollary} (Riemann-Hodge bilinear relations for Grothendieck pairing on $\Omega_f$) \cite{R}
Assume $f:\mathbb{C}^{n+1} \to \mathbb{C}$ is a holomorphic germ with isolated singularity. Suppose $\mathfrak{f}$ is the corresponding map to $N$ on $H^n(X_{\infty})$, via the isomorphism $\Phi$. Define 

\[ P_l=PGr_l^W:=\ker(\mathfrak{f}^{l+1}:Gr_l^W\Omega_f \to Gr_{-l-2}^W\Omega_f) \]

\vspace{0.5cm}

\noindent
Going to $W$-graded pieces;
\begin{equation}
\widehat{Res}_l: PGr_l^W \Omega_f \otimes_{\mathbb{C}} PGr_l^W \Omega_f \to \mathbb{C}
\end{equation}

\vspace{0.5cm}

\noindent
is non-degenerate and according to Lefschetz decomposition 

\[ Gr_l^W\Omega_f=\bigoplus_r \mathfrak{f}^r P_{l-2r} \]

\vspace{0.5cm}

\noindent
we will obtain a set of non-degenerate bilinear forms,

\begin{equation}
\widehat{Res}_l \circ (id \otimes \mathfrak{f}^l): P Gr_l^W \Omega_f  \otimes_{\mathbb{C}} P Gr_l^W \Omega_f  \to \mathbb{C}, 
\end{equation} 

\vspace{0.1cm}

\begin{center}
$\widehat{Res}_l=res_{f,0}\ (id \otimes \tilde{C} .\  \mathfrak{f}^l)$
\end{center}

\vspace{0.5cm}

\noindent
where $\tilde{C}$ is as in 3.2, such that the corresponding hermitian form associated to these bilinear forms is positive definite. In other words, 

\vspace{0.5cm}

\begin{itemize}

\item $\widehat{Res}_l(x,y)=0, \qquad x \in P_r, \ y  \in P_s, \ r \ne s $

\item If $x \ne 0$ in $P_l$, 

\[ Const \times res_{f,0}\ (C_lx,\tilde{C} .\  \mathfrak{f}^l .\bar{x})>0  , \ \ \ \ \ \ \ \ Const \in \mathbb{C}\]

\noindent
where $C_l$ is the corresponding Weil operator, and the conjugation is as in (7). 

\end{itemize}

\vspace{0.5cm}

\end{corollary}

\noindent
Note that the map 

\[ A_f=\dfrac{\mathcal{O}_X}{\partial f} \to \Omega_f ,\qquad f \mapsto fdx_0...dx_n \]

\vspace{0.5cm}

\noindent
is an isomorphism. Thus, the above corollary would state similarly for $A_f$. 

\vspace{0.5cm}

\section{Chern character (Denis trace map) for isolated hypersurface singularities}

\vspace{0.5cm}

The Hochschild chain complex of $MF(R,f)$ is quasi-isomorphic to the Koszul complex of the regular sequence $\partial_0 f,..., \partial_n f$. In particular the Hochschild homology (and also the Hochschild cohomology) of 2-periodic dg-category $MF(R,f)$ is isomorphic to the module of relative differentials or the Jacobi ring of $f$, \cite{D}. 
 
\vspace{0.5cm}

\begin{theorem} (T. Dykerhoff) \cite{D}, \cite{PV} The canonical bilinear form on the Hochschild homology of category of matrix factorizations $\mathcal{C}=MF(P,f)$ of $f$, after the identification
\begin{equation}
HH_*MF(P,f) \cong A_f \otimes dx [n]
\end{equation}

\vspace{0.5cm}

\noindent
coincides with
\begin{equation}
\langle g \otimes dx, h \otimes dx \rangle =(-1)^{n(n-1)/2}res_{f,0}(g,h)
\end{equation}
\end{theorem}

\vspace{0.5cm}

\noindent
The chern character or Denis trace map is a ring homomorphism 
\begin{equation}
ch:K'_0(X) \to HH_0(X) \cong \Omega_f
\end{equation}

\vspace{0.5cm}

\noindent
where $K_0^{\prime}$ is free abelian group on the isomorphism classes of finitely generated modules modulo relations obtained from short exact sequences. The construction of the chern character map or chern classes is functorial w.r.t flat pull back. In the special case of $i:X \hookrightarrow Y$ the compactification, the following diagram commutes,

\begin{equation}
\begin{CD}
K_0^{\prime}(Y_0) @>ch_Y>> HH_0(Y_0) \cong \Omega_f^Y @>\Phi_Y^{-1}>> H^n(Y_{\infty})\\
@Vi^*VV         @VVi^*V          @VVi^*V\\
K_0^{\prime}(X_0) @>>ch_X> HH_0(X_0) \cong \Omega_f^X @>>\Phi_X^{-1}> H^n(X_{\infty}).
\end{CD}
\end{equation}

\vspace{0.5cm}

\noindent
Given a matrix factorization $(A,B)$ for a maximal Cohen-Macaulay $M$, one can find de Rham representatives for the chern classes. Consider $\mathbb{C}[[x_0,...,x_n]]$ as a $\mathbb{C}[[t]]$-module with $t$ acting as multiplication by $f$. Denote by $\Omega^p$
the module of germs of $p$-forms on $\mathbb{C}^{n+1}$, and let $\Omega_f^p=\Omega^p/(df \wedge \Omega^{p-1})$. One puts $\omega(M)=dA \wedge dB$. The components of the chern character

\begin{equation}
ch_M:=tr(\exp(\omega(M)))= \sum_i \frac{1}{i!}{\omega}^i(M)
\end{equation}

\vspace{0.5cm}

\noindent
are well-defined classes

\begin{equation}
\omega^i(M)=tr((dA \wedge dB)^i) \in \Omega_f^{2i}/(df \wedge \Omega^{2i-1})
\end{equation}

\vspace{0.5cm}

\noindent
There are also odd degree classes

\vspace{0.5cm}

\begin{center}
$\eta^i(M):=tr(AdB(dA \wedge dB)^i) \in \Omega_f^{2i+1}/{\Omega_f^{2i}}$.
\end{center}

\vspace{0.5cm}

\noindent
The group $\Omega_f^{2i+1}/d\Omega_f^{2i}$ can be identified with the cyclic homology $HC_i(P/\mathbb{C}\{t\})$. They fit into the following short exact sequence such that $d\eta^{i-1}=\omega^i(M)$.

\vspace{0.5cm}

\begin{center}
$0 \to \Omega_f^{2i-1}/\Omega_f^{2i-2} \to \Omega^{2i}/(df \wedge \Omega^{2i-1}) \to \Omega^{2i}/\Omega^{2i-1} \to 0$.
\end{center}

\vspace{0.5cm}

\noindent
If the number of variables $n+1$ is even, then a top degree form sits in the Brieskorn module

\vspace{0.5cm}

\begin{center}
$\mathcal{H}_f^{(0)}=\Omega^n/(df \wedge d\Omega^{n-1})$
\end{center}

\vspace{0.5cm}

\noindent
a free $\mathbb{C}[[t]]$-module of rank $\mu$. The higher residue pairing 

\[ K:\mathcal{H}_f^{(0)} \times \mathcal{H}_f^{(0)} \to \mathbb{C}[t, t^{-1}] \]

\vspace{0.5cm}

of K. Saito can be seen as the de Rham realization of the Seifert form of the singularity, \cite{BVS}.

\vspace{0.5cm}

\section{Positivity of Theta pairing-Main Result}

\vspace{0.5cm}

The following theorem was conjectured in \cite{MPSW}.

\vspace{0.5cm}

\begin{theorem} Let $S$ be an isolated hypersurface singularity of dimension $n$. If $n$ is odd, then $(-1)^{(n+1)/2}\Theta$ is positive semi-definite on $G(R)_\mathbb{Q}$, i.e $(-1)^{(n+1)/2}\Theta(M,M) \geq 0$. 
\end{theorem}

\begin{proof} By additivity of $\Theta$ on each variable, we may replace $M,N$ by maximal Cohen-Macaulay modules. According to 4.1 and 3.2 the determination of the sign of $\Theta$ amounts to understanding how the image of chern classes look like in the MHS of $\Omega_f$. By Theorem 3.2 it amounts to the same things for the image in $H^n(X_{\infty})$ under the isomorphism $\Phi$. The following diagram is commutative by the functorial properties of chern character.

\begin{equation}
\begin{CD}
K_0^{\prime}(Y_0) @>\Phi_Y^{-1} \circ ch_Y>> H^n(Y_{\infty})\\
@Vi^*VV                   @VVi^*V\\
K_0^{\prime}(X_0) @>>\Phi_X^{-1} \circ ch_X> H^n(X_{\infty}).
\end{CD}
\end{equation}

\vspace{0.5cm}

\noindent
We are assuming that $i^*$ is surjective. By what was said, the chern class we are concerned with, is a Hodge cycle. The commutativity of the above diagram allows us to replace the pre-image of the chern character for $X$, with similar cycle upstairs. Because the polarization form $S_X$ was defined via that of $S_Y$. Thus, if 

\[ H^n(Y_{\infty})=\oplus_{p+q=n} H^{p,q} \]

\vspace{0.5cm}

\noindent
is the Hodge decomposition, the only non trivial contribution in the cup product will be for the $H^{n/2,n/2}$, and the polarization form is evidently definite on this subspace (Hodge cycles). Because the map $N_Y$ is of type $(-1,-1)$ for the Hodge structure of $H^n(Y_{\infty})$ and the polarization $S_Y( H^{n/2,n/2}, H^{n/2-1,n/2-1})=0$ for obvious reasons, the corresponding chern class should lie in $H_{\ne 1}^n$. In this way one only needs to prove the positivity statement for Hochster $\Theta$ when the chern character is in $H_{Y,\ne 1}$, and this is the content of Theorem 2.8.
\end{proof}

\end{document}